\documentclass{article}
\usepackage{amsmath}
\usepackage{amssymb}
\usepackage{bm}

\DeclareBoldMathCommand\boldlangle{\left\langle}
\DeclareBoldMathCommand\boldrangle{\right\rangle}
\input xypic
\input xy
\xyoption{all}
\newtheorem{theorem}{Theorem}

\newtheorem{definition}[theorem]{Definition}

\newtheorem{proposition}[theorem]{Proposition}

\newenvironment{proof}[1][Proof]{\noindent\textbf{#1.} }{\ \rule{0.5em}{0.5em}}
\author{Hatice G\"{u}ls\"{u}n Akay \ \  Ummahan Ege Arslan\
}

\vspace{0.5cm}
\title{\textbf{On the Barr exactness property of $\mathsf{BXMod/R}$}}
\date{}

\begin{document}
\maketitle

{\noindent\textbf{Abstract}} \ \ In this work, it is shown that the category $\mathsf{BXMod/R}$ of
braided crossed modules over a fixed commutative algebra $R$ is an exact category in the sense of Barr. \\
{\noindent\textbf{MR(2010) Subject Classification}} \ \  18A30, 18E10. 

\section{Introduction}

Exact categories for additive categories were introduced by Quillen \cite%
{Quillen}. Barr defined exact categories for non-additive categories \cite{Barr}. He introduced exact categories in order to define a good
notion of non-abelian cohomology.

We will mention regular categories closely related by exact categories. There exist in the literature many different definitions of regular categories, which are all equivalent under the assumptions that finite limits and coequalisers exist. We will recall the following definition from \cite{Tholen}, but a weaker version of it is given in \cite{Barr}.

A category $C$ is called regular if it satisfies the following three properties:

\begin{description}
\item[i)] $C$ is finitely complete,

\item[ii)] If $f:X\longrightarrow Y$ is a morphism in $C$, and
$$ \xymatrix@R=30pt@C=30pt{
Z \ar[d]_-{p_1} \ar[r]^-{p_0}
                &X \ar[d]^-{f}  \\
  X  \ar[r]_-{f}
                & Y }
$$
\noindent is a pullback $($%
then $Z\underset{p_{1}}{\overset{p_{0}}{\rightrightarrows }}X$ is called the kernel pair of $f)$, the coequaliser of $%
p_{0},p_{1}$ exists;

\item[iii)] If $f:X\longrightarrow Y$ is a morphism in $C$, and
$$ \xymatrix@R=30pt@C=30pt{
W \ar[d]_-{g} \ar[r]^-{}
                &X \ar[d]^-{f}  \\
 Z \ar[r]_-{}
                & Y }
$$
\noindent is a pullback, and if $f$ is a regular epimorphism, then $g$ is a regular epimorphism as well.

\end{description}

If a regular category additionally has the property that every equivalence relation is effective that is every equivalence
relation is a kernel pair, then it is called a Barr exact category.
\smallskip

Examples of exact categories are:

$\left( 1\right) $ The category $Set$ of sets.

$\left( 2\right) $ The category of non empty sets.

$\left( 3\right) $ Any abelian category.

$\left( 4\right) $ Every partially ordered set considered as a category.

$\left( 5\right) $ For any small category $C$, the functor category $\left(
\mathcal{C}^{op},Set\right) $.

Brown and Gilbert introduced in \cite{Brown1} the notion of braided regular
crossed module of groupoids and groups as an algebraic model for homotopy
3-types equivalent to Conduch\'{e}'s $2$-crossed module. The reduced case of
braided regular crossed module is called a braided crossed module of groups.
Braided crossed modules in the category of commutative algebras were defined by
Ulualan in \cite{Ulualan}.

The purpose of this paper is to answer the question whether the category of
braided crossed modules of commutative algebras is exact. We prove that the category of braided crossed modules of commutative algebras is an exact category.

\begin{quote}
\textbf{Conventions}
\end{quote}

Throughout this paper $k$ will be a fixed commutative ring with $0\neq 1$.
All $k$-algebras will be commutative and associative.

\section{Braided crossed modules}

In order to give the notion of braided crossed modules of commutative algebras, we will
recall the concept of crossed modules of commutative algebras. Crossed modules of groups
originate in algebraic topology and more particularly in homotopy theory. Mac Lane and Whitehead showed in \cite{Mac} that crossed modules of groups modelled homotopy 2-types (3-types in their notation). The commutative algebra case of crossed modules is contained in
the paper of Lichtenbaum and Schlessinger \cite{LS} and also in the work of
Gerstenhaber \cite{Gers} under different names. Some categorical results and
Koszul complex link are also given by Porter \cite{Porter}.

\begin{definition}\cite{Porter}
A crossed module of commutative algebras, $(C,R,\partial)$, is an $R$-algebra, $C$, together with an $R$-algebra morphism $\partial :C\longrightarrow
R $ such that for all $c,c^{\prime }\in C$
\begin{equation*}
 \text{ }\partial \left(
c\right) \cdot c^{\prime }=cc^{\prime }
\end{equation*}%
where $R$ is a $k$-algebra.
A morphism of crossed modules from $\left( C,R,\partial \right) $ to $\left( C^{\prime },R^{\prime
},\partial ^{\prime }\right) $ is a pair of $k$-algebra morphisms, $\phi :C\longrightarrow C^{\prime }$
and $\psi :R\longrightarrow R^{\prime }$ such that
\begin{equation*}
\left( i\right) \text{ }\partial ^{\prime }\phi =\psi \partial \text{ \
\ \ and \ \ \ }\left( ii\right) \text{ }\phi \left( r\cdot c\right)
=\psi \left( r\right) \cdot \phi \left( c\right)
\end{equation*}%
for all $r\in R$ and $c\in C$. We thus get the category $\mathsf{XMod}$ of
crossed modules.

There is, for a fixed algebra $R$, a subcategory $\mathsf{XMod/R}$ of the category of crossed modules, which has as objects those crossed modules with $R$ as the
\textquotedblleft base\textquotedblright , i.e., all $\left( C,R,\partial
\right) $ for this fixed $R$, and having as morphisms from $\left(
C,R,\partial _{1}\right) $ to $\left( C^{\prime },R,\partial _{2}\right) $
just those $\left( f_{1},f_{0}\right) $ in $\mathsf{XMod}$ in which $%
f_{0}:R\longrightarrow R$ is the identity homomorphism on $R$.
\end{definition}

\begin{definition}\cite{Ulualan}
A braided crossed module of commutative algebras $\partial :C\rightarrow R$ is a
crossed module with the braiding function $\left\{ -,-\right\} :R\times
R\longrightarrow C$ satisfying the following axioms:

$%
\begin{array}{ll}
BCM1) & \partial \left\{ r,r^{\prime }\right\} =rr^{\prime },%
\end{array}%
$

$%
\begin{array}{ll}
BCM2) & \left\{ \partial c,\partial c^{\prime }\right\} =cc^{\prime },%
\end{array}%
$

$%
\begin{array}{lllll}
BCM3) & \left\{ \partial c,r\right\} =r\cdot c & \text{and} & \left\{
r,\partial c\right\} =r\cdot c,&
\end{array}%
$

$%
\begin{array}{ll}
BCM4) & \left\{ rr^{\prime },r^{\prime \prime }\right\} -\left\{ r,r^{\prime
}r^{\prime \prime }\right\} =0,%
\end{array}%
$

{\noindent}for all $r,r^{\prime },r^{\prime \prime }\in R$ and $c,c^{\prime
}\in C$.

We denote such a braided crossed module of commutative algebras by $\left\{
C,R,\partial \right\} $.

If $\left\{ C,R,\partial _{1}\right\} $ and $\left\{ C^{\prime },R^{\prime
},\partial _{2}\right\} $ are braided crossed modules, a morphism$\ $%
\begin{equation*}
\left( f_{1},f_{0}\right) :\left\{ C,R,\partial _{1}\right\} \longrightarrow
\left\{ C^{\prime },R^{\prime },\partial _{2}\right\} ,
\end{equation*}%
of braided crossed modules is given by a morphism of crossed modules such
that
\begin{equation*}
\left\{ -,-\right\} \left( f_{0}\times f_{0}\right) =f_{1}\left\{
-,-\right\} .
\end{equation*}

We thus get the category $\mathsf{BXMod}$ of braided crossed modules of commutative algebras.

In the case of a morphism $\left( f_{1},f_{0}\right) $ between braided
crossed modules with the same base $R$, i.e. where $f_{0}$ is the identity on $R
$ with $f_{1}\left\{ -,-\right\} =\left\{ -,-\right\} $, then we say that $%
f_{1}$ is a morphism of braided crossed $R$-modules, $\partial _{1}:C\longrightarrow R$ is a braided crossed $R$-module and we use $\left\{C,\partial _{1} \right\} $ instead of $\left\{C,R,\partial _{1} \right\} $. This gives a
subcategory $\mathsf{BXMod/R}$ of $\mathsf{BXMod}$. Our results are obtained for this subcategory.
\end{definition}

Several well known examples of crossed modules give rise to braided crossed modules as follows.

\textbf{Examples of braided crossed modules}

\noindent $\left( 1\right) $ Any identity map of $k$-algebras $\partial
:X\longrightarrow X$ is a braided crossed module with $\left\{ x,y\right\}
=xy$.

\noindent $\left( 2\right) $ If $C$ is a $k$-algebra and $C^{2}$ is an ideal generated by $\left\{ c_{1}c_{2}\mid c_{1},c_{2}\in C\right\} $. Then $\partial :C^{2}\longrightarrow C$ is a braided
crossed module with $\left\{ c_{1},c_{2}\right\} =c_{1}c_{2}$, for $%
c_{1},c_{2}\in C$.

\noindent $\left( 3\right) $ Any $R$-module $M$ can be considered as an $R$%
-algebra with zero multiplication and hence the zero morphism $%
0:M\longrightarrow R$ is a braided crossed module with $\left\{ r,r^{\prime
}\right\} =0$.

\noindent $\left( 4\right) $ Let $\left\{ C,R,\partial _{1}\right\} $ and $%
\left\{ C^{\prime },R^{\prime },\partial _{2}\right\} $ be two braided
crossed modules, then $\left\{ C\times C^{\prime },R\times R^{\prime
},\partial \right\} $ is a braided crossed module.

\section{The Barr exactness property of braided crossed modules}

Our aim now is to obtain that $\mathsf{BXMod/R}$ is an exact category. For
this purpose, we have to prove some statements in this section.
\begin{proposition} \label{prop3}
In $\mathsf{BXMod/R}$ every pair of morphisms with common domain and
codomain has an equaliser.
\end{proposition}

\begin{proof}
Let $f,g:\left\{ C,\partial \right\} \longrightarrow \left\{ D,\delta
\right\} $ be two morphisms of braided crossed $R$-modules. Let $E$ denotes the
set $E=\left\{ c\in C\mid f\left( c\right) =g\left( c\right) \right\} .$ It can be easily checked that $\left\{ E,\varepsilon \right\} $ has the structure of
a braided crossed $R$-module, and the inclusion $u:\left\{ E,\varepsilon
\right\} \longrightarrow \left\{ C,\partial \right\} $ is a morphism of
braided crossed $R$-modules and clearly $fu=gu$.

Suppose that there exist a braided crossed $R$-module, $\left\{ E^{\prime
},\varepsilon ^{\prime }\right\} $ and a morphism $u^{\prime }:\left\{
E^{\prime },\varepsilon ^{\prime }\right\} \longrightarrow \left\{
C,\partial \right\} $ of braided crossed $R$-modules such that $fu^{\prime
}=gu^{\prime }.$ Then for all $x\in E^{\prime },$ $f\left( u^{\prime }\left(
x\right) \right) =g\left( u^{\prime }\left( x\right) \right) ,$ and hence $%
u^{\prime }\left( x\right) \in E.$ Thus, we have $\alpha :\left\{ E^{\prime
},\varepsilon ^{\prime }\right\} \longrightarrow \left\{ E,\varepsilon
\right\} $ by $\alpha \left( x\right) =u^{\prime }\left( x\right) $ from which we get
\begin{equation*}
\varepsilon \alpha \left( x\right) =\varepsilon u^{\prime }\left( x\right)
=\partial u^{\prime }\left( x\right) =\varepsilon ^{\prime }\left( x\right),
\end{equation*}
\begin{equation*}
\alpha \left( r\cdot x\right) =u^{\prime }\left( r\cdot x\right) =r\cdot
u^{\prime }\left( x\right)=r\cdot \alpha \left( x\right),
\end{equation*}
for all $x \in E^{\prime }, r \in R$. Since $u$, $u^{\prime }$ are braided crossed $R$-module morphisms, it is
clear that
$\alpha \left\{ r,r^{\prime }\right\} =\left\{ r,r^{\prime
}\right\} $ for all $r, r^{\prime } \in R$.

Let $\alpha ^{\prime }:E^{\prime }\longrightarrow E$ be a morphism of
braided crossed $R$-module such that $u\alpha ^{\prime }=u^{\prime }.$ Since
$\alpha \left( x\right) =u^{\prime }\left( x\right)
=u\alpha ^{\prime }\left( x\right) =\alpha ^{\prime }\left( x\right)$ for all $x\in E^{\prime }$,
we have that $\alpha$ is the unique morphism which makes the diagram
\begin{equation*}
 \xymatrix@R=40pt@C=40pt{
\left\{ E,\varepsilon\right\} \ar@{^{(}->}[r]^-{u}& \left\{ C,\partial \right\} \ar@{->}@<2pt>[r]^-{f}
\ar@{->}@<-2pt> [r]_-{g} &  \left\{ D,\delta\right\} \\ \left\{ E^{\prime},\varepsilon ^{\prime }\right\} \ar@{.>}[u]_{\alpha} \ar[ur]_{u^{\prime }} }
\end{equation*}

\noindent commutative. Hence $u$ is the equaliser of $\left( f,g\right),$ as required.
\end{proof}

\begin{proposition} \label{prop4}
$\mathsf{BXMod/R}$ has finite products.
\end{proposition}

\begin{proof}
Let $\left\{ C,\partial \right\} $ and $\left\{ D,\delta \right\}$ be braided crossed $R$-modules. The product $ \left\{ C\sqcap D,\tau \right\} $ is the pullback over the
terminal object $\left\{ R,i_{R}\right\} ,$
\begin{equation*}
 \xymatrix@R=40pt@C=40pt{
  \left\{ C\sqcap D,\tau \right\} \ar[d]_{\delta ^{\prime }} \ar[r]^{\partial ^{\prime }}
                &\left\{ C,\partial \right\}  \ar[d]^{\partial}  \\
  \left\{ D,\delta \right\}  \ar[r]_{\delta}
                & \left\{ R,i_{R}\right\}             }
\end{equation*}
\noindent where $C\sqcap D=\left\{ \left( c,d\right) \mid \partial \left( c\right) =\delta
\left( d\right) \right\} $ and $\tau :C\sqcap D\longrightarrow R$ is defined
by $\tau=\partial \partial^{\prime }=\delta \delta ^{\prime }.$ Then by induction, $\mathsf{BXMod/R}$
has finite products.
\end{proof}

\begin{proposition}
\label{crll} $\mathsf{BXMod/R}$ is finitely complete.
\end{proposition}

\begin{proof}
Follows from Propositions \ref{prop3} and \ref{prop4}.
\end{proof}
\begin{proposition}
\label{prp2} In $\mathsf{BXMod/R}$ every morphism has a kernel pair, and the
kernel pair has a coequaliser.
\end{proposition}

\begin{proof}
Let $\left\{ A,\partial \right\} $ and $\left\{ B,\beta \right\} $ be two
braided crossed $R$-modules. Let $f:\left\{ A,\partial \right\}
\longrightarrow \left\{ B,\beta \right\} $ be a morphism of braided crossed $%
R$-modules. Then $\left( A,f\right) $ is a crossed $B$-module, where $B$
acts on $A$ via $\beta $ and the homomorphism $\alpha :A\times _{B}A\longrightarrow B$ defined by $\alpha \left( a,a^{\prime }\right) =f\left( a\right) =f\left(
a^{\prime }\right) $ is a crossed $B$-module, where $A\times _{B}A=\left\{
\left( a,a^{\prime }\right) \mid f\left( a\right) =f\left( a^{\prime
}\right) \right\} $ is a $B$-algebra:

$%
\begin{array}{ll}
& \alpha \left( b\cdot \left( a,a^{\prime }\right) \right) =f\left( b\cdot
a\right) =b\cdot f\left( a\right) =b\cdot \alpha \left( a,a^{\prime }\right),
\end{array}%
$ \smallskip

$%
\begin{array}{ll}
& \alpha \left( a,a^{\prime }\right) \cdot \left( a_{1},a_{1}^{\prime
}\right) =\left( f\left( a\right) \cdot a_{1},f\left( a^{\prime }\right)
\cdot a_{1}^{\prime }\right) =\left( a,a^{\prime }\right) \left(
a_{1},a_{1}^{\prime }\right),
\end{array}%
$

{\noindent }for all $\left( a,a^{\prime }\right) ,\left( a_{1},a_{1}^{\prime
}\right) \in A\times _{B}A, b \in B.$

We can define $\partial ^{\prime }:A\times _{B}A\longrightarrow R$ by $%
\partial ^{\prime }\left( a,a^{\prime }\right) =\partial \left( a\right)
=\partial \left( a^{\prime }\right) $, since $\beta \alpha =\beta f=\partial
$. It is easily checked that $\left( A\times _{B}A,R,\partial ^{\prime }\right)
$ is a crossed module:

$%
\begin{array}{ll}
& \partial ^{\prime }\left( r\cdot \left( a,a^{\prime }\right) \right)
=\partial \left( r\cdot a\right) =r\cdot \partial \left( a\right) =r\cdot
\partial ^{\prime }\left( a,a^{\prime }\right),
\end{array}%
$

$%
\begin{array}{llllll}
& \partial ^{\prime }\left( a,a^{\prime }\right) \cdot \left(
a_{1},a_{1}^{\prime }\right)  & = & \partial \left( a\right) \cdot \left(
a_{1},a_{1}^{\prime }\right)  & = & \left( \partial \left( a\right) \cdot
a_{1},\partial \left( a\right) \cdot a_{1}^{\prime }\right)  \\
&  &  &  & = & \left( \partial \left( a\right) \cdot a_{1},\partial \left(
a^{\prime }\right) \cdot a_{1}^{\prime }\right)  \\
&  &  &  & = & \left( aa_{1},a^{\prime }a_{1}^{\prime }\right)  \\
&  &  &  & = & \left( a,a^{\prime }\right) \left( a_{1},a_{1}^{\prime
}\right),
\end{array}%
$
{\noindent }for all $\left( a,a^{\prime }\right) ,\left( a_{1},a_{1}^{\prime
}\right) \in A\times _{B}A$, $r \in R$.

Below we will show that $\left\{ A\times _{B}A,\partial ^{\prime }\right\} $ is a braided
crossed $R$-module with the braiding map
\begin{equation*}
\left\{ -,-\right\} :R\times R\longrightarrow A\times _{B}A
\end{equation*}%
defined by $\left\{ r,r^{\prime }\right\} =\left( \left\{ r,r^{\prime }\right\}
,\left\{ r,r^{\prime }\right\} \right) $,

$%
\begin{array}{ll}
\mathbf{BCM1)} & \partial ^{\prime }\left\{ r,r^{\prime }\right\}=\partial \left\{ r,r^{\prime }\right\}=rr^{\prime },%
\end{array}%
$

$%
\begin{array}{llll}
\mathbf{BCM2)} & \left\{ \partial ^{\prime }\left( a,a^{\prime }\right)
,\partial ^{\prime }\left( b,b^{\prime }\right) \right\}  & = & \left(
\left\{ \partial a,\partial b\right\} ,\left\{ \partial a^{\prime },\partial
b^{\prime }\right\} \right)  \\
&  & = & \left( ab,a^{\prime }b^{\prime }\right)  \\
&  & = & \left( a,a^{\prime }\right) \left( b,b^{\prime }\right),
\end{array}%
$

$%
\begin{array}{lrll}
\mathbf{BCM3)} & \left\{ \partial ^{\prime }\left( a,a^{\prime }\right)
,r\right\}  & = & \left( \left\{ \partial a,r\right\} ,\left\{ \partial
a^{\prime },r\right\} \right)  \\
&  & = & \left( r\cdot a,r\cdot a^{\prime }\right)  \\
&  & = & r\cdot \left( a,a^{\prime }\right)  \\
& \left\{ r,\partial ^{\prime }\left( a,a^{\prime }\right) \right\}  & = &
\left( \left\{ r,\partial a\right\} ,\left\{ r,\partial a^{\prime }\right\}
\right)  \\
&  & = & \left( r\cdot a,r\cdot a^{\prime }\right)  \\
&  & = & r\cdot \left( a,a^{\prime }\right),
\end{array}%
$

$%
\begin{array}{llll}
\mathbf{BCM4)} & \left\{ rr^{\prime },r^{\prime \prime }\right\} -\left\{
r,r^{\prime }r^{\prime \prime }\right\}  & = & \left( \left\{ rr^{\prime
},r^{\prime \prime }\right\} ,\left\{ rr^{\prime },r^{\prime \prime
}\right\} \right) -\left( \left\{ r,r^{\prime }r^{\prime \prime }\right\}
,\left\{ r,r^{\prime }r^{\prime \prime }\right\} \right)  \\
&  & = & \left( \left\{ rr^{\prime },r^{\prime \prime }\right\} -\left\{
r,r^{\prime }r^{\prime \prime }\right\} ,\left\{ rr^{\prime },r^{\prime
\prime }\right\} -\left\{ r,r^{\prime }r^{\prime \prime }\right\} \right)
\\
&  & = & \left( 0,0\right),
\end{array}%
$\bigskip

{\noindent }for all $\left( a,a^{\prime }\right) ,\left( b,b^{\prime }\right) \in
A\times _{B}A$, $r,r^{\prime },r^{\prime \prime }\in R$.
\\The following diagram
$$ \xymatrix@R=40pt@C=40pt{
 A\times _{B}A\ar@{->}@<2pt>[r]^-{p_1}
\ar@{->}@<-2pt> [r]_-{p_2} \ar[d]_-{\partial ^{\prime }}& \ A \ar[r]^-{f} \ar[d]^-{\partial}
& B\ar[d]^-{\beta} \\R\ar@{=}[r]^{} & R\ar@{=}[r]^{}  & R }
$$
{\noindent }commutes and the morphisms $p_{1}$ and $p_{2}$ above are
morphisms of braided crossed $R$-modules.
This construction satisfies
universal property: Let $%
\{E,\delta\} $ be a braided crossed $R$-module and $p_{1}^{\prime
},p_{2}^{\prime }:\{E,\delta\}\longrightarrow \{A,\partial
\} $ be any morphisms of braided crossed $R$-modules with $%
fp_{1}^{\prime }=fp_{2}^{\prime },$ then there exist a unique morphism
\begin{equation*}
h:\{E,\delta\} \longrightarrow \{
A\times _{B}A,\partial ^{\prime }\}
\end{equation*}%
given by $h\left( e\right) =\left( p_{1}^{\prime }\left( e\right)
,p_{2}^{\prime }\left( e\right) \right) ,$ for all $e\in E,$ which makes the diagram
$$ \xymatrix@R=35pt@C=35pt{
\{E,\delta\} \ar@/_/[ddr]_{p_{2}^{\prime }}
  \ar@/^/[drr]^{p_{1}^{\prime }}
    \ar@{.>}[dr]|-{h}                   \\
   &\{A\times _{B}A,\partial ^{\prime }\} \ar[d]^-{p_{2}}
   \ar[r]_-{p_{1}}
                      & \{A,\partial\} \ar[d]^-{f}    \\
   & \{A,\partial\}\ar[r]_-{f}     &\{B,\beta\}            } $$
{\noindent }commutative.
\\Then $\left( p_{1},p_{2}\right) $ is the kernel pair of
the morphism $f$.

Now we will show that the pair $\left( p_{1},p_{2}\right) $ has a
coequaliser. Let $I$ be an ideal of $A$ generated by
all the elements of the form $p_{1}\left( x\right) -p_{2}\left( x\right) ,$
for all $x=\left( a,a^{\prime }\right) \in A\times _{B}A$. We will define the braided crossed $R$-module $\delta
:A/I\longrightarrow R$ by $\delta \left( a+I\right) =\partial \left(
a\right) ,$ $\left\{ r,r^{\prime }\right\} =\left\{ r,r^{\prime }\right\} +I$
for $a\in A,$ $r,r^{\prime }\in R$. Since $I\subseteq
Ker\partial $, $\delta $ is well defined. By the definition of $\left\{
A/I,\delta \right\} $, the morphism $q:\left\{ A,\partial \right\}
\longrightarrow \left\{ A/I,\delta \right\} $ is the induced projection and
it is a morphism of braided crossed $R$-modules, i.e., the diagram
$$ \xymatrix@R=40pt@C=40pt{
A\times _{B}A \ar@{->}@<2pt>[r]^-{p_1}
\ar@{->}@<-2pt> [r]_-{p_2} \ar[dr]_-{\partial ^{\prime }}& \ A \ar[r]^-{q} \ar[d]^-{\partial}&  A/I\ar[dl]^-{\delta} \\ & R & }
$$
{\noindent }commutes.
Suppose there exist a braided crossed $R$-module $\left\{ A^{\prime
},\alpha ^{\prime }\right\} $ and a morphism of braided crossed $R$-modules $%
q^{\prime } :\left\{ A,\partial \right\}
\longrightarrow \left\{ A^{\prime },\alpha ^{\prime }\right\} $ such that $%
q^{\prime }p_{1}=q^{\prime }p_{2},$ then there exists a unique morphism $
\varphi :\left\{ A/I,\alpha \right\} \longrightarrow \left\{
A^{\prime },\alpha ^{\prime }\right\} $ defined by $\varphi \left( a+I\right)
=q^{\prime }\left( a\right) $, satisfying $\varphi q=q^{\prime },$ i.e. $q$
is the universal among all the morphisms $q^{\prime } :\left\{ A,\partial \right\} \longrightarrow \left\{ A^{\prime },\alpha
^{\prime }\right\} $ for any braided crossed $R$-module $\left\{ A^{\prime
},\alpha ^{\prime }\right\} .$ So we get the following commutative diagram:
$$ \xymatrix@R=40pt@C=40pt{
A\times _{B}A \ar@{->}@<2pt>[r]^-{p_1}
\ar@{->}@<-2pt> [r]_-{p_2} \ar[dr]_-{\partial ^{\prime }}& \ A \ar[r]^-{q} \ar[d]^-{\partial} \ar[dr]|(0.7){q^{\prime }}&  A/ I\ar[dl]|(0.7){ \ \ \ {\delta} \ \ \ } \ar@{.>}[d]^{\varphi}\\ & R & A^{\prime }\ar[l]^{\delta ^{\prime }} }
$$
Then $q$ is the coequaliser of the pair $\left( p_{1},p_{2}\right) .$

Therefore in $\mathsf{BXMod/R}$, the kernel pair of every morphism exists
and has a coequaliser.
\end{proof}

\begin{proposition}
\label{prp1}In $\mathsf{BXMod/R}$ every regular epimorphisms are stable under pullback.
\end{proposition}
\begin{proof}
Let $\Phi :\left\{ C,\mu \right\} \longrightarrow \left\{ D,\sigma \right\} $
be a regular epimorphism in $\mathsf{BXMod/R,}$ which means that $%
C\longrightarrow D$ is a regular epimorphism of $k$-algebras. Let $\eta
:\left\{ G,\theta \right\} \longrightarrow \left\{ D,\sigma \right\} $ be
any morphism in $\mathsf{BXMod/R.}$ The pullback of $\Phi $ along $\eta $ is
$\Phi ^{\ast }:\left\{ C\times _{D}G,\rho \right\} \longrightarrow \left\{
G,\theta \right\} $ given by $\Phi ^{\ast }\left( c,g\right) =g$ where $%
C\times _{D}G=\left\{ \left( c,g\right) \mid \Phi \left( c\right) =\eta
\left( g\right) \right\} $ and $\rho :C\times _{D}G\longrightarrow R$ is a
braided crossed $R$-module defined by $\rho \left( c,g\right) =\mu \left(
c\right) =\theta \left( g\right) $ with $\left\{ r,r^{\prime }\right\}
=\left( \left\{ r,r^{\prime }\right\} ,\left\{ r,r^{\prime }\right\} \right)
$, for all $r,r^{\prime }\in R,$ $c\in C,$ $g\in G.$ Since in the category
of $k$-algebras the regular epimorphisms are characterised as the surjective
homomorphisms, these are closed in this way under pullback. Thus $\Phi
^{\ast }$ is a surjective homomorphism.
\begin{equation*}
\xymatrix@R=40pt@C=40pt{
\left\{ C\times
_{D}G,\rho \right\} \ar[d]_{\Phi ^{\ast }} \ar[r]^{}
                &\left\{ C,\mu \right\} \ar[d]^{\Phi}  \\
  \left\{ G,\theta \right\}  \ar[r]_{\eta}
                & \left\{ D,\sigma \right\} }
\end{equation*}
\noindent We claim that the surjective morphism $\Phi
^{\ast }$ is a regular epimorphism, that is $\Phi ^{\ast }$ is the coequaliser of a pair of morphisms. Define
\begin{equation*}
E=\left\{ \left( x,y\right) \in \left( C\times _{D}G\right) \times \left(
C\times _{D}G\right) \mid \Phi^{\ast } \left( x\right) =\Phi^{\ast } \left( y\right)
\right\}
\end{equation*}%
and let $\left\{ E,\alpha \right\} \underset{q}{\overset{p}{\rightrightarrows }}%
\left\{ C\times _{D}G,\rho \right\} $ be the first and second projections. Since $G$ is isomorphic to the quotient of $C\times
_{D}G$, $\Phi ^{\ast }$ is the coequaliser of $p
$ and $q$. Thus we get that if $\Phi $ is regular epimorphism, so is $\Phi ^{\ast }$, as required.
\end{proof}

\begin{theorem}
\label{theo} $\mathsf{BXMod/R}$ is regular.
\end{theorem}
\begin{proof}
The proof is a direct consequence of Propositions \ref{crll}, \ref{prp2} and \ref{prp1}.
\end{proof}

Now we recall the definition of equivalence relation from \cite{Wells}.
\begin{definition}
Let $A$ be a category with finite limits. If $A$ is an object, a subobject $%
\left( d^{0},d^{1}\right) :E\longrightarrow A\times A$ is called an
equivalence relation if it is
\noindent $\mathbf{ER1.}$ Reflexive: There is an arrow $r:A\longrightarrow E$ such that $%
d^{0}r=d^{1}r=id_{A};$

\noindent $\mathbf{ER2.}$ Symmetric: There is an arrow $s:E\longrightarrow E$ such that $%
d^{0}s=d^{1}$ and $d^{1}s=d^{0};$

\noindent $\mathbf{ER3.}$ Transitive: If
\begin{equation*}
\xymatrix@R=30pt@C=30pt{
 T \ar[d]_{q_{2}} \ar[r]^{q_{1}}
                &E \ar[d]^{d^{0}}  \\
  E \ar[r]_{d^{1}}
                & A }
\end{equation*}
\noindent is a pullback, there is an arrow $t:T\longrightarrow E$ such that $%
d^{1}t=d^{1}q_{1}$ and $d^{0}t=d^{0}q_{2.}$
\end{definition}

\begin{proposition}
\label{prp3} Every equivalence relation
$$ \xymatrix@R=40pt@C=40pt{
\{ E,\partial\}\ar@{->}@<2pt>[r]^-{u}
\ar@{->}@<-2pt> [r]_-{v}& \ \{A,\alpha\} }
$$

{\noindent }in the category of braided crossed $R$-modules is effective.%

\end{proposition}

\begin{proof}
Let $A/E$ be the set of all equivalence classes $\left[ a\right] $ with
respect to $E,$ i.e., $\left[ a\right] =\left\{ b\in A\mid \left( a,b\right)
\in E\right\} .$ $A/E$ has the structure of an $R$-algebra. $\overline{%
\alpha }:A/E\longrightarrow R$ induced by $\alpha $ is well defined since $%
\alpha u=\alpha v.$ We can form the braided crossed $R$-module $\overline{%
\alpha }$ with braiding map $\left\{ r,r^{\prime }\right\} =\left[ \left\{
r,r^{\prime }\right\} \right] $ and get the following diagram,
$$ \xymatrix@R=40pt@C=40pt{
E \ar@{->}@<2pt>[r]^-{u}
\ar@{->}@<-2pt> [r]_-{v} \ar[dr]_-{\partial}& \ A \ar[r]^-{q} \ar[d]^-{\alpha}
& A/ E\ar[dl]^-{\overline{\alpha }} \\ & R & }
$$

{\noindent }of morphisms of braided crossed $R$-modules, where $q$ is the
projection onto $A/E$. By the definition of an equivalence relation on $A$, we have $\left( u\left(
x\right) ,v\left( x\right) \right) \in E,$ for $x\in E$. Since
\begin{equation*}
E\subseteq A\times _{R}A=\left\{ \left( a,a^{\prime }\right) \mid \alpha
\left( a\right) =\alpha \left( a^{\prime }\right) \right\} ,
\end{equation*}

{\noindent we get }$\alpha u\left( x\right) =\alpha v\left( x\right) $, thus
$\left( 0,u\left( x\right) -v\left( x\right) \right) \in E,$ therefore $qu=qv
$. Suppose that there exist a braided crossed $R$-module $\left\{ D,\omega
\right\} $ with $u^{\prime },v^{\prime }:\left\{ D,\omega \right\}
\longrightarrow \left\{ A,\alpha \right\} $ such that $qu^{\prime
}=qv^{\prime },$ so $\left[ u^{\prime }\left( d\right) \right] =\left[
v^{\prime }\left( d\right) \right] ,$ i.e. $\left( u^{\prime }\left(
d\right) ,v^{\prime }\left( d\right) \right) \in E,$ and therefore there
exists a unique morphism $\theta :\left\{ D,\omega \right\} \longrightarrow
\left\{ E,\partial \right\} ,$ such that the diagram
$$ \xymatrix@R=35pt@C=35pt{
 \{D,\omega\}\ar@/_/[ddr]_{v^{\prime }} \ar@/^/[drr]^{u^{\prime }}
    \ar@{.>}[dr]|-{ \ \ {\theta} \ \ }                   \\
   & \{E,\partial\}\ar[d]^-{v} \ar[r]_-{u}
                      & \{A,\alpha\} \ar[d]^-{q}    \\
   & \{A,\alpha\}\ar[r]_-{q}     & \{A/E,\overline{\alpha }\}
   } $$

{\noindent }commutes.

Thus $\left( u,v\right) $ is the kernel pair of a morphism $q:\left\{
A,\alpha \right\} \longrightarrow \left\{ A/E,\overline{\alpha }\right\} $
in the category of braided crossed $R$-modules.
\end{proof}

\begin{theorem}
The category $\mathsf{BXMod/R}$ of braided crossed $R$-modules is a Barr exact category.\label{th}
\end{theorem}

\begin{proof}
The conditions of exact category for $\mathsf{BXMod/R}$ are satisfied by Theorem \ref{theo} and Proposition \ref{prp3}
, which completes the proof.
\end{proof}\newline

\textbf {Acknowledgments}

The first author was partially supported by T\"{U}B\.{I}TAK (The Scientific and
Technological Research Council Of Turkey).

\noindent  H. G\"{u}ls\"{u}n Akay and U. Ege Arslan \newline
Department of Mathematics and Computer Science \newline
Eski\c{s}ehir Osmangazi University \newline
26480 Eski\c{s}ehir/Turkey\newline
e-mails: \{hgulsun, uege\} @ogu.edu.tr


\begin{thebibliography}{9}
\bibitem{Barr} M. Barr, \ {Exact Categories, in: Exact Categories and Categories of Sheaves.}
Springer Lecture Notes in Math, $236,1-120,1971$.

\bibitem{Wells} M. Barr, C. Wells, \ {Toposes, Triples and Theories.}
Theory and Applications of Categories, $12,1-288,2005$.

\bibitem{Brown1} R. Brown, N.D. Gilbert, \ {Algebraic Models of }${3}${%
-Types and Automorphism Structures for Crossed Modules.} Proc. London Math.
Soc., $59,51-73,$ $1989$.

\bibitem{Gers} M. Gerstenhaber, \ {On the Deformation of Rings and Algebras.%
} Ann. Math., $84,1966$.

\bibitem{LS} S. Lichtenbaum, M. Schlessingger, \ {The Cotangent Complex of
a Morphism.} Trans. American Society, $18,41-70,1967$.

\bibitem{Mac} S. Mac Lane, J.H.C. Whitehead, \ {On the 3-type of a Complex.} Proc. Nat. Acad. Sci., $41-48,1950$.

\bibitem{Tholen} M. C. Pedicchio, W. Tholen , \ {Categorical Foundations. Special Topics in Order, Topology, Algebra, and Sheaf Theory,}
Mathematics and Its Applications, Brics Lecture Series, $97$, $2004$.

\bibitem{Porter} T. Porter, \ {Some Categorical Results in the Category of
Crossed Modules in Commutative Algebra.} J.Algebra, $109,415-429,1978$.

\bibitem{Quillen} D. Quillen, \ {Higher Algebraic K-Theory I.} Higher K-Theories, Lecture Notes in Mathematics, Springer, $341$, $85-147,1973$.

\bibitem{Ulualan} E. Ulualan, \ {De\u{g}i\c{s}meli Cebirler \"{U}zerinde
Kuadratik Mod\"{u}ller,} Ph.D. Thesis, ESOGU FBE, $2004$.

\end{thebibliography}
\end{document}